\documentclass[12pt,twoside]{article}

\usepackage[a4paper]{geometry}
\makeatletter
\renewcommand\title[1]{\gdef\@title{\reset@font\Large\bfseries #1}}
\renewcommand\section{\@startsection {section}{1}{\z@}%
                                   {-3.5ex \@plus -1ex \@minus -.2ex}%
                                   {2.3ex \@plus.2ex}%
                                   {\normalfont\large\bfseries}}
\renewcommand\subsection{\@startsection{subsection}{2}{\z@}%
                                     {-3ex\@plus -1ex \@minus -.2ex}%
                                     {1.5ex \@plus .2ex}%
                                     {\normalfont\normalsize\bfseries}}
\renewcommand\subsubsection{\@startsection{subsubsection}{3}{\z@}%
                                     {-2.5ex\@plus -1ex \@minus -.2ex}%
                                     {1.5ex \@plus .2ex}%
                                     {\normalfont\normalsize\bfseries}}

\def\@runningauthor{}\newcommand{\runningauthor}[1]{\def\runningauthor{#1}}
\def\@runningtitle{}\newcommand{\runningtitle}[1]{\def\runningtitle{#1}}

\renewcommand{\ps@plain}{%
\renewcommand{\@evenfoot}{\footnotesize \hfill\thepage}
\renewcommand{\@oddfoot}{\footnotesize  \hfill\thepage}
\renewcommand{\@evenhead}{\footnotesize\scshape \hfill\runningauthor\hfill}
\renewcommand{\@oddhead}{\footnotesize\scshape \hfill\runningtitle\hfill}}
\g@addto@macro\bfseries{\boldmath}

\makeatother

\usepackage{amsthm,amsmath,amssymb}

\usepackage{graphicx}

\usepackage[colorlinks=true,citecolor=black,linkcolor=black,urlcolor=blue]{hyperref}

\theoremstyle{plain}
\newtheorem{theorem}{Theorem}
\newtheorem{lem}[theorem]{Lemma}
\newtheorem{cor}[theorem]{Corollary}
\newtheorem{prop}[theorem]{Proposition}

\theoremstyle{definition}
\newtheorem{defn}[theorem]{Definition}
\newtheorem{exmp}[theorem]{Example}

\theoremstyle{remark}
\newtheorem{rem}[theorem]{Remark}

\title{Characterisation of the parameters of maximum weight spectrum codes according to their spread}

\runningtitle{Characterisation of the parameters of MWS codes according to their spread}

\author{Alessio Meneghetti\\
\small Department of Mathematics\\[-0.8ex]
\small University of Trento\\[-0.8ex] 
\small Trento, Italy\\
\small\tt alessio.meneghetti@unitn.it\\
\and
Wrya K. Kadir \\
\small Department of Informatics\\[-0.8ex]
\small University of Bergen\\[-0.8ex]
\small Bergen, Norway\\
\small\tt wrya.kadir@uib.no
}

\runningauthor{A.\ Meneghetti, W.\ K.\ Kadir }

\date{}

\begin{document}

\maketitle

\thispagestyle{empty}

\begin{abstract}
We introduce the concept of spread of a code, and we specialize it to the case of maximum weight spectrum (MWS) codes. We classify two newly-defined sub-families of MWS codes according to their weight distributions, and completely describe their fundamental parameters. We focus on one of these families, the strictly compact MWS codes, proving their optimality as MWS codes and linking them to known codes.
\end{abstract}


\section{Introduction}
\label{sec: intro}

Let $\mathbb{F}_q$ be a finite field with $q$ elements. A $[n,k]_q$ linear code is a $k$-dimensional subspace of $\mathbb{F}_q^n$. The (Hamming) weight $w(x)$ of a vector $x\in \mathbb{F}_q^n$ is defined to be the number of nonzero components of $x$. The minimum of weights where $x\neq 0$ is the minimal distance $d$ of the code. A linear code $[n,k]_q$ whose the minimum distance is $d$ shall be denoted by $[n,k,d]_q$. A generator matrix for an $[n,k]_q$ linear code $C$ is a $k\times n$ matrix over $\mathbb{F}_q$ whose row vectors generate $C$. Let $(n+1)$-tuples $\{A_0,A_1,\ldots,A_n\}$ be the weight distribution of an $[n,k,d]_q$ linear code $C$ where $A_i$ is the number of codewords in $C$ with weight $i$. We denote by $S(C)$ the set of non-zero weights of a linear code $C$, i.e. $s\in S(C)$ if there exists $c\in C\setminus\{0\}$ such that $\mathrm{w}(c)=s$.

The weight distribution of linear codes has been appeared in many works over the years. Mac Williams in \cite{MacWIlliam63} constructed linear codes by employing the elements of the given weight set. He also proved that the weight set of a linear code can be derived from the weight set of its dual. Delsarte in \cite{DELSARTE1973} discussed some properties of codes using their weight distributions. In particular an upper bound $$|C|\leq \sum_{i=0}^{r}\binom{n}{i}(q-1)^i$$ on the size of a $q$-ary code $C$ with $r$ nonzero distinct weights was proposed. In addition, a lower bound on the size of $C$ was obtained using the number of distinct nonzero weights in its dual code $C^{\perp}$. More bounds on $|C|$ can be found in \cite{Enomoto-distance}.

The concept of binary linear codes with distinct weights has been firstly proposed and fully characterised in \cite{haily2015automorphism}. In this work, the authors define a \textit{distinct weight} (DW) code as a binary code whose weight distribution is of the form $A_i\in\{0,1\}$ for each $i$. Other than providing an infinite family of DW codes they study the automorphism group of DW codes.
\begin{defn}\label{def: binary MWS}
We denote by $B_k$ the $[2^k-1,k,1]_2$ DW code defined by the generator matrix
$$G=
\begin{bmatrix}
1_1^{(k)}\\
1_2^{(k)}\\
\vdots\\
1_k^{(k)}
\end{bmatrix}
$$
where $1_{i}^{(k)}$ for $i=1,\ldots,k$ is the $(2^k-1)$-bits row vector whose first $2^i-1$ bits are equal $1$, the remaining are equal 0, e.g. 
$$
\left\{
\begin{array}{c}
1_1^{(3)}=(1,0,0,0,0,0,0)\\
1_2^{(3)}=(1,1,1,0,0,0,0)\\
1_3^{(3)}=(1,1,1,1,1,1,1)
\end{array}
\right.
\Rightarrow
\begin{bmatrix}
1_1^{(3)}\\
1_2^{(3)}\\
1_3^{(3)}
\end{bmatrix}
=
\begin{bmatrix}
1&0&0&0&0&0&0\\
1&1&1&0&0&0&0\\
1&1&1&1&1&1&1
\end{bmatrix}
.$$  
\end{defn}
Due to \cite{haily2015automorphism}, there exist binary linear codes in which no two distinct codewords have the same Hamming weight and they are called distinct weight codes. Recently, independently from \cite{haily2015automorphism}, the authors of \cite{shi2019many} proposed the problem of determining the maximum possible number of distinct weights in a block code over any finite field $\mathbb{F}_q$. Other than providing a complete solution in the general nonlinear $q$-ary case, they showed a construction for binary linear codes attaining the maximum possible number of distinct weights. This family coincides with the family $B_k$ recalled in Definition \ref{def: binary MWS}. In the same work, the authors proposed a bound and conjectured that the maximum number of distinct nonzero weights that a $k$-dimensional $q$-ary code can have is $\frac{q^k-1}{q-1}$. This work opened several lines of research, starting from two independent works providing solutions for the conjecture \cite{alderson-neri2019maximum}, \cite{meneghetti2018linear}. Shorter codes with maximum number of weights for a given dimension were later discussed in \cite{alderson2019note,Cohen-tilhuizen}. In \cite{Shi-sole-cyclic}, the authors discussed the largest number of nonzero weights a cyclic code of dimension $k$ over $\mathbb{F}_q$ can have. 

\begin{theorem}\cite{shi2019many}\label{theorem: max number distinct weights}
Let $C$ be a $k$-dimensional $q$-ary linear code. Then the maximum possible number $q_k$ of distinct non-zero weights in $C$ equals $q_k=\frac{q^k-1}{q-1}$.
\end{theorem}
Following the notation established in \cite{alderson-neri2019maximum}, these codes are now known as \textit{maximum weight spectrum} (MWS) codes. Observe that DW codes are a particular case of MWS codes.

We recall another important example of MWS code. 
\begin{defn}\label{def: q 2 MWS}
We denote with $D_q$ the $[\frac{q(q+1)}{2},2,\frac{q(q-1)}{2}]_q$ code generated by 
$$G'=
\begin{bmatrix}
\underbrace{\begin{matrix}1\\\alpha^1\end{matrix}}_1&\underbrace{\begin{matrix}1&1\\\alpha^2&\alpha^2\end{matrix}}_{2}&\begin{matrix}\cdots\\\end{matrix}& \underbrace{\begin{matrix}1&\cdots&1\\\alpha^{q-1}&\cdots&\alpha^{q-1}\end{matrix}}_{q-1\;\mathrm{times}}&\underbrace{\begin{matrix}1&\cdots&1\\0&\cdots&0\end{matrix}}_{q\;\mathrm{times}}
\end{bmatrix} ,
$$
where $\alpha$ is a primitive element of $\mathbb{F}_q$. 
\end{defn}
Observe that for $k=2$ and $q=2$, the codes $B_k$ and $D_{q}$ coincide. This is a hint that there exists a family $\mathcal{SC}_{q,k}$ of MWS codes sharing the same structure and properties.

Subsequent works propose several constructions for new families of codes \cite{alderson-neri2019maximum,Cohen-tilhuizen,patanker2019weight}, trying in particular to obtain tight bounds on the minimum possible length of MWS codes. Another related open question is whether, given two integers $k$ and $s$, there exists a linear code with dimension $k$ and exactly $s$ distinct weights. A positive answer has been provided in \cite{meneghetti2018linear} for the particular case of binary codes, while the $q$-ary case is left as a conjecture. 

Due to \cite{alderson-neri2019maximum}, the minimum possible length of an $[n,k,d]_q$ MWS code is $n=\frac{q}{2}\cdot\frac{q^k-1}{q-1}$. An $[n,k,d]_q$ MWS code achieves this length if it has  codewords of every Hamming weight from $d$ to $n$ and we call it \textit{strictly compact} (Definition \ref{def: compact}) MWS code. An $[m,l,s]_q$ \textit{compact} (Definition \ref{def: compact}) MWS code has codewords of every Hamming weight from $s$ to $s+\frac{q^l-1}{q-1}-1$.

In this work we investigate the parameters of MWS codes according to new classifications. The properties and parameters of MWS codes, strictly compact MWS codes and compact MWS codes are discussed. In particular, the codes $B_k$ and $D_q$ are strictly compact, and, as all strictly compact codes, they are optimal MWS codes (we will prove this in Corollary \ref{cor: strictly optimal}).  We introduce the concept of \textit{spread} (Definition \ref{def: spread}) of a code, a sort of distance between the weight distribution of a code from a reference one.  We use strictly compact codes as a reference code to measure the spread of MWS codes, namely, the spread of a MWS code $C$ is the distance between its weight distribution and the one of a hypothetical strictly compact MWS code with equal length (see Definition \ref{def: spread MWS}). This choice is based upon the optimality of strictly compact codes, and it turns out that the spread is deeply linked with the length of MWS. Moreover, we investigate bounds on the minimum distance and the spread of known MWS codes.

\section{Notation and remarks}
\label{sec: notation}

In this section we discuss parameters of MWS codes according to their weight distribution. In particular we will see how length, dimension and minimum distance of MWS codes are related to their weight distribution. Particular emphasis will be put on the parameters of the family $\mathcal{SC}_{q,k}$ briefly introduced in Section \ref{sec: intro}.
 The following lemma is a direct consequence of the definition of linear codes, hence we recall it without providing a proof.
\begin{lem}\label{lem: link n weights}
Let $\{A_i\}_{i\in\{0,\ldots,n\}}$ be the weight distribution of an $[n,k,d]_q$ linear code $C$. If $C$ has no coordinate identically equal to $0$ (\textit{non-degenerate}), then 
$$n=\dfrac{\sum_i{A_i}i}{q^{k}-q^{k-1}}.$$
\end{lem}

We introduce some useful definition to classify MWS codes according to their weight distribution.
\begin{defn}\label{def: compact}
An MWS $[n,k,d]_q$ code $C$ is \textit{compact} if its set of weights $S(C)$ is
$
S(C)=\{d,d+1,\ldots,d+{q}_k-1\}
$
 where $q_k=\frac{q^k-1}{q-1}$, and it is \textit{strictly compact} if it is compact and $n\in S(C)$.
\end{defn}
\begin{prop}\label{prop-example SCMWS}
$B_{k}$ and $D_{q}$ in Definitions \ref{def: binary MWS} and \ref{def: q 2 MWS} are strictly compact for respectively any choice of $k$ and $q$. In particular 
$$
S(B_k)=\{1,\ldots,n\}=\{\;\underbrace{1\;,\;2\;,\;\ldots\;,\;2^k-2\;,\;2^k-1}_{{q}_k\;\mathrm{consecutive}\;\mathrm{integers}}\;\} ,
$$
and 
$$
S(D_q)=\{d,\ldots,n\}=\bigg\{\;\underbrace{\frac{q(q-1)}{2}\;,\;\ldots\;,\;\frac{q(q+1)}{2}}_{{q}_k\;\mathrm{consecutive}\;\mathrm{integers}}\;\bigg\} .
$$
\end{prop}
\begin{proof}
It follows from computation. 
\end{proof}
Equivalent to Definition \ref{def: compact}, an MWS is compact if its weight distribution is
$$\{A_i\}_{i\in\{0,\ldots,n\}}=\{1,\;\underbrace{0,\ldots,0}_{1\leq i\leq d-1}\;,\;\underbrace{q-1,\ldots,q-1}_{d\leq i\leq d+q_k-1}\;,\;\underbrace{0,\ldots,0}_{i>d+q_k-1}\}$$
while it is strictly compact if
$$\{A_i\}_{i\in\{0,\ldots,n\}}=\{1,\;\underbrace{0,\ldots,0}_{1\leq i\leq d-1}\;,\;\underbrace{q-1,\ldots,q-1}_{i\geq d}\;\}.$$
Observe that for a strictly compact MWS we can write  
$$
S(C)=\{d,d+1,\ldots,n\}=\left\{n-{q}_k+1, \ldots, n-1,n\right\}=\{n-i \mid i=0,1,\ldots,{q}_k-1\}\;,
$$
and in the general case, we can write the set of weights of an MWS code as
\begin{equation}\label{eq: def sc}
S(C)=\{n-s_0,\ldots,n-s_{{q}_k-1}\} .
\end{equation}
A strictly compact MWS code with $d=1$ is called \textit{full weight spectrum} (FWS) code in \cite{alderson2019note}.
\begin{defn}\label{def: spread}
Let $C$ and $\bar{C}$ be two $q$-ary codes with equal length $n$ and dimension $k$. We define the \textit{spread} of $C$ w.r.t a target code $\bar{C}$ as the value
$$
\Delta_{\bar{C}}(C)=\frac{1}{q-1}\sum_{i=1}^M\left(\bar{w}_i-w_i\right)\;,
$$
where $\{w_1,\ldots,w_M\}$ and  $\{\bar{w}_1,\ldots,\bar{w}_M\}$ are the set of all non-zero weights of $C$ and $\bar{C}$.
\end{defn}
We remark that $\Delta_{\bar{C}}(C)$ is therefore equal to $\frac{1}{q-1}\left(\sum_{i=1}^ni\bar{A}_i-\sum_{j=1}^njA_j\right)$, with $\bar{A}_i$ and $A_j$ the weight distributions of  $\bar{C}$ and $C$ respectively.  Observe that $\Delta_{\bar{C}}(C)=-\Delta_C(\bar{C})$.
\\
For our purposes, we specialize Definition \ref{def: spread} to the case of MWS codes, using as a target code a hypothetical strictly compact MWS code. In this case we can therefore omit to specify the target code, since given the $[n,k]_q$ MWS code $C$ we know exactly the weight distribution of a strictly compact $[n,k]_q$ MWS code.

\begin{defn}\label{def: spread MWS}
The \textit{spread} of an MWS code $C$ with $S(C)=\{n-s_0,\ldots,n-s_{{q}_k-1}\}$ is the value
\begin{equation}\label{eq: def delta}
\Delta(C)=(s_0-0)+(s_1-1)+\cdots+ (s_{{q}_k-1}-{q}_k+1)=\sum_{i=0}^{{q}_k-1}(s_i-i) .
\end{equation}
\end{defn}
With this definition, $\Delta(C)$ can be thought as a measure of how much the weight distribution of an MWS code is spread across the entire set $\{1,\ldots,n\}$, in terms of the distance from the weight distribution of a hypothetical strictly compact MWS code of length $n$. Due to Definition \ref{def: spread MWS}, we can equivalently define a strictly compact MWS code $C$ as an MWS code with spread $\Delta(C)=0$. As we will see, strictly compact MWS codes are optimal codes, namely, there exist no MWS code with equal dimension and length strictly less than their length. To prove this claim, the next section deals with the characterisation of the parameters of strictly compact MWS codes and their comparison with general MWS codes.

\section{Strictly Compact MWS codes}
In this section we use the notation introduced in previous section to discuss the parameters of MWS codes, with a focus on strictly compact codes. We prove in particular the link between length  and spread of MWS codes, implying the optimality of strictly compact codes.
\begin{theorem}
\label{theorem: parameters}
Let $\mathcal{SC}_{q,k}$ be a strictly compact MWS code of dimension $k$ over $\mathbb{F}_q$. Then its parameters are 
$$
\left[\;\frac{q}{2}\;{q}_k\;,\;k\;,\;\left(\frac{q}{2}-1\right){q}_k+1\;\right]_q .
$$
\end{theorem}
\begin{proof}
A strictly compact $[n,k,d]_q$ MWS code $C$ is by definition a code with spread $\Delta(C)=0$. In particular 
\begin{equation}\label{eq: sc}
S(C)=\{n, n-1, \ldots, n-{q}_k+1\} ,\;\;\mathrm{and}\;\;  A_i=q-1, i\in S(C). 
\end{equation}

By Lemma \ref{lem: link n weights} we have $n=\frac{\sum_i{A_i}i}{q^{k}-q^{k-1}}$, which, due to Equation \eqref{eq: sc}, can be written as
$$
n=
\frac{
\sum_{i=0}^{{q}_k-1}{(q-1)}\left(n-i\right)
}{
q^{k}-q^{k-1}
}
=
\frac{
n\cdot {q}_k-\sum_{i=0}^{{q}_k-1}i
}{
q^{k-1}
}
=
\frac{
n\cdot {q}_k-\frac{\left({q}_k-1\right){q}_k}{2}
}{
q^{k-1}
} ,
$$
hence
\begin{equation}\label{eq: proof n strictly compact}
n\left( {q}_k-q^{k-1}\right)=\frac{\left({q}_k-1\right){q}_k}{2}
\end{equation}
We observe that on the left-hand side of Equation \eqref{eq: proof n strictly compact}, the coefficient of $n$  is ${q}_k-q^{k-1}={q}_{k-1}$, while on the right-hand side we have ${q}_k-1=q\cdot{q}_{k-1}$. By substitution, we deduce
\begin{equation}\label{eq: proof n strictly last}
n{q}_{k-1}=\frac{\left(q\cdot{q}_{k-1}\right){q}_k}{2} ,
\end{equation}
which leads to the claimed length of $C$. \\
The proof that $d=\left(\frac{q}{2}-1\right){q}_k+1$ is then a direct consequence of $C$ being a strictly compact MWS code, since in this case the minimum weight is $n-{q}_k+1$.
\end{proof}
\begin{exmp}
As expected, if we use $q=2$, then the parameters are $[2^k-1,k,1]_2$, namely the parameters of the known code $\mathcal{SC}_{2,k}$. Similarly, if we use $k=2$, the parameters listed in Theorem \ref{theorem: parameters} become $\left[\;\frac{q}{2}\;{q}_2\;,\;2\;,\;\left(\frac{q}{2}-1\right){q}_2+1\;\right]_q=\left[\;\frac{q}{2}\;(q+1)\;,\;2\;,\;\frac{q}{2}(q-1)\;\right]_q$, i.e. the parameters of $\mathcal{SC}_{q,2}$.
\end{exmp}
\begin{exmp}
Consider $q=4$ and $k=3$, namely the smallest case with $q$ even which are not covered by $\mathcal{SC}_{2,k}$ or $\mathcal{SC}_{q,2}$. If a strictly compact MWS code would exists, then it would be a $[42,3,22]_4$ code.
\end{exmp}

\begin{theorem}
\label{theorem: length spread}
Let $C$ be an MWS code. Then $n=\frac{q}{2}{q}_k+\frac{\Delta(C)}{{q}_{k-1}}.$
\end{theorem}
\begin{proof}We proceed similarly to the proof of Theorem \ref{theorem: parameters}, where Equation \eqref{eq: sc} is substituted by Equation \eqref{eq: def sc}. This implies, after some computation, that Equation \eqref{eq: proof n strictly compact} becomes
$$
n\left( {q}_k-q^{k-1}\right)=\frac{\left({q}_k-1\right){q}_k}{2}+
\sum_{i=0}^{{q}_k-1}\left(s_i-i\right).
$$
We recall that by Equation \eqref{eq: def delta} the sum on the right-hand side is $\Delta(C)$. In this way, instead of Equation \eqref{eq: proof n strictly last} we have  
$
n{q}_{k-1}=\frac{\left(q\cdot{q}_{k-1}\right){q}_k}{2} +\Delta(C)  .
$
\end{proof}

\begin{cor}\label{cor: strictly optimal}
Strictly compact MWS codes are optimal among MWS codes.
\end{cor}
\begin{proof}
It follows from Theorem \ref{theorem: parameters} and Theorem \ref{theorem: length spread}. The spread of an MWS code is indeed equal to zero if and only if the code is strictly compact, and the length of an MWS code grows together with its spread.
\end{proof}


\section{On the parameters of MWS codes}
In this section we consider again general MWS codes, and we use strictly compact MWS codes to obtain a characterization of their parameters.
\begin{cor}
Let $C$ be an MWS code with $q$ odd and odd dimension $k$. Then $\Delta(C)>0$.
In particular, there does not exist a strictly compact MWS code $\mathcal{SC}_{q,k}$ for $q$ and $k$ both odd.
\end{cor}
\begin{proof}
If $\Delta(C)=0$ then $C$ is a strictly compact MWS code, hence by Theorem  \ref{theorem: parameters} we know that its length is 
$
n=\frac{q}{2}{q}_k\in\mathbb{Z}.$ 
We have two possible cases: either $q$ is even, or ${q}_k$ is. Observe that if $q$ is odd, the latter is true if and only if $k$ is even.
\end{proof}
\begin{cor}\label{cor: qk odd}
Let $C$ be an $[n,k,d]_q$ MWS code.
\begin{enumerate}
\item\label{cor qk case 1}  Let $q\cdot k$ be even. Then $\Delta(C)=h{q}_{k-1}$, with $h$ a positive integer.
\item\label{cor qk case 2} Let $q\cdot k$ be odd. Then $\Delta(C)=\frac{2h+1}{2}{q}_{k-1}$. 
\\
In particular, $n\ge \frac{q}{2}{q}_k+\frac{1}{2}$ and its spread is at least $\frac{1}{2}{q}_{k-1}$.
\end{enumerate}
\end{cor}

\begin{proof}
The length of $C$ is equal to $\frac{q}{2}{q}_k+\frac{\Delta(C)}{{q}_{k-1}}\in\mathbb{Z}$.
\begin{enumerate}
\item If $q\cdot k$ is even, then either $q$ is even or ${q}_k$ is. This implies that $\frac{q}{2}{q}_k\in\mathbb{Z}$, and since $n\in\mathbb{Z}$, so has to be $n-\frac{q}{2}{q}_k=\frac{\Delta(C)}{{q}_{k-1}}$. 
\item If both $q$ and $k$ are odd, then $\frac{q}{2}{q}_k$ is not an integer. However, since $n\in\mathbb{Z}$,  $\frac{\Delta(C)}{{q}_{k-1}}$ is equal to $\frac{2h+1}{2}$ for a non-negative integer $h$. This implies that the spread of $C$ is $\Delta(C)=\frac{2h+1}{2}{q}_{k-1}$.
\end{enumerate}
\end{proof}

\begin{exmp}
Consider an MWS code $C$ of dimension $k=3$ over $\mathbb{F}_3$, so that by Corollary \ref{cor: qk odd}, $C$ has spread $\Delta(C)=\frac{2h+1}{2}{q}_{k-1}=2(2h+1)$. In particular $C$ cannot be strictly compact, since its spread is at least $2$.\\
We suppose now that $\Delta(C)=2$, hence by Theorem \ref{theorem: length spread} its length is $\frac{q}{2}{q}_k+\frac{2}{{q}_{k-1}}=\frac{3}{2}\frac{3^3-1}{3-1}+2\frac{3-1}{3^2-1}=\frac{39}{2}+\frac{1}{2}=20$. The hypothetical set of weights of a strictly compact MWS code with same length would have been
$$
\{8,9,10,11,12,13,14,15,16,17,18,19,20\} .
$$
In our case we instead know that the spread $\Delta(C)=2$, hence we have two possibilities, either
$$
\{6,9,10,11,12,13,14,15,16,17,18,19,20\} \text{ or }
\{7,8,10,11,12,13,14,15,16,17,18,19,20\} 
.$$

We already said that in the general case $\Delta(C)=2(2h+1)$, we assume now that $h=1$ so that the spread is 6. In this case we proceed as above and apply Theorem  \ref{theorem: length spread} to deduce that the length of $C$ is 
$$
n=\frac{q}{2}\frac{q^k-1}{q-1}+\frac{q-1}{q^{k-1}-1}6=\frac{3}{2}\frac{3^3-1}{2}+\frac{2}{3^{2}-1}\cdot 6=\frac{3}{2}\frac{26}{2}+\frac{3}{2}=\frac{39}{2}+\frac{3}{2}=21.
$$
The only possible sets of weights of such an MWS code are

\begin{align*}
\{& 3,10,11,12,13,14,15,16,17,18,19,20,21\}\\
\{& 4,9,11,12,13,14,15,16,17,18,19,20,21\}\\
\{& 5,8,11,12,13,14,15,16,17,18,19,20,21\}\\
\{& 6,7,11,12,13,14,15,16,17,18,19,20,21\}\\
\{& 5,9,10,12,13,14,15,16,17,18,19,20,21\}\\
\{& 6,8,10,12,13,14,15,16,17,18,19,20,21\}\\
\{& 7,8,9,12,13,14,15,16,17,18,19,20,21\}.\\
\end{align*}

\end{exmp}

\begin{prop}\label{distMWS-prop1}
Let $C$ be an $[n,k,d]_q$ MWS code.
\begin{enumerate}
\item If either $q$ or $k$ are even, then 
$$
\left\{
\begin{array}{l}
n=\frac{q}{2}{q}_k+h\\
d\ge \left(\frac{q}{2}-1\right){q}_k+1+h\left(1-{q}_{k-1}\right)\\
d\leq\left(\frac{q}{2}-1\right){q}_k+1+h-\left\lceil h \frac{{q}_{k-1}}{{q}_k}\right\rceil 
\end{array}
\right.
$$
where $h$ is a non-negative integer.
\item If both $q$ and $k$ are odd, then 
$$
\left\{
\begin{array}{l}
n=\frac{q}{2}{q}_k+h+\frac{1}{2}\\
d\ge \left(\frac{q}{2}-1\right){q}_k+\frac{3}{2}+h\left(1-{q}_{k-1}\right)\\
d\leq\left(\frac{q}{2}-1\right){q}_k+\frac{3}{2}+h-\left\lceil h \frac{{q}_{k-1}}{{q}_k}\right\rceil 
\end{array}
\right.
$$
where $h$ is a non-negative integer.
\end{enumerate}
\end{prop}
\begin{proof}
It follows from Corollary \ref{cor: qk odd}, noticing that the minimum distance is linked to the spread of the code, which in turn depends on $h$. 
\end{proof}

\section{Compact MWS codes}
In this section we focus on the parameters of compact MWS codes, which are MWS codes whose set of weights is of the form $\{d,d+1,\ldots,d+{q}_k-1\}$.\\
Let us start with considering a code with either  $q$ odd and $k$ even or with $q$ even, i.e. as in Corollary \ref{cor: qk odd}, case \ref{cor qk case 1}, so that the spread is $h{q}_{k-1}$ and the length is $n=\frac{q}{2}{q}_k+h$.
\\
If we suppose that $A_n=0$, then the spread has to be at least ${q}_k$, and this implies $h\ge q$. As a consequence $n\ge \frac{q}{2}{q}_k+q$.
\\
Suppose now the maximum weight in $C$ is $n-j$, namely $A_{n-j}\neq 0$ and $A_i=0\;$ for any $i>n-j$. Similarly to above, $h{q}_{k-1}\ge j {q}_k$, implying $\frac{h}{j}\ge q$. In this case we have obtained that $n\ge \frac{q}{2}{q}_k+jq$. We obtain the following complete characterization of the parameters of compact MWS codes.

\begin{cor}\label{cor: parameters compact}
Consider an $[n,k,d]_q$ compact MWS code $C$. Then 
\begin{enumerate}
\item if $q\cdot k$ is even, then the parameters of $C$ are
$$
\left[\frac{q}{2}{q}_k+jq,k,\left(\frac{q}{2}-1\right){q}_k+j(q-1)+1\right]_q\;,
$$
for a non-negative integer $j$. The set of weights of the code is
$$
S(C)=\left\{\left(\frac{q}{2}-1\right){q}_k+j(q-1)+1\;\quad,\; \ldots\;,\quad \left(\frac{q}{2}+1\right){q}_k+j(q-1)\right\}\;.
$$
\item if $q\cdot k$ is odd, then the parameters of $C$ are
$$
\left[\frac{q}{2}{q}_k+jq+\frac{1}{2},k,\left(\frac{q}{2}-1\right){q}_k+j(q-1)+\frac{3}{2}\right]\;,
$$
for a non-negative integer $j$. The set of weights of the code is
$$
S(C)=\left\{\left(\frac{q}{2}-1\right){q}_k+j(q-1)+\frac{3}{2}\;\quad,\; \ldots\;,\quad \left(\frac{q}{2}+1\right){q}_k+j(q-1)+\frac{1}{2}\right\}\;.
$$
\end{enumerate}
\end{cor}
\begin{proof}
We start from the case $q\cdot k$ that is even, which was already introduced in the discussion above. We consider a code with maximum weight $n-j$, and we assume it is compact. Then the length is  $n= \frac{q}{2}{q}_k+jq$ and therefore the distance is
$$
d=\frac{q}{2}{q}_k+jq-j-{q}_k+1=\left(\frac{q}{2}-1\right){q}_k+j(q-1)+1.
$$
The case $q\cdot k$ that is odd, namely as in Corollary \ref{cor: qk odd}, case \ref{cor qk case 2}, is very similar, since if the maximum weight is $n-j$, the spread is $h{q}_{k-1}+\frac{1}{2}\ge j{q}_k$, then $h\ge jq-\frac{1}{2}$ and it gives $h\ge jq$. As a consequence, the length is $n\ge \frac{q}{2}{q}_k+jq+\frac{1}{2}$, and  this bound is attained with equality if the code is compact.
 The minimum distance is therefore
$$
d=\frac{q}{2}{q}_k+jq+\frac{1}{2}-j-{q}_k+1=\left(\frac{q}{2}-1\right){q}_k+j(q-1)+\frac{3}{2}.
$$

\end{proof}
As a consequence of Corollary \ref{cor: parameters compact} we have the following result.
\begin{cor}
There does not exist compact non-binary MWS codes with distance $1$.
\end{cor}


\section{Known codes}

MWS codes have been studied in  \cite{alderson-neri2019maximum,meneghetti2018linear}, where the authors also presented some bounds on the length of $[n,k]_q$ MWS codes. In this section we investigate bounds on the minimum distance and the spread of known MWS $[n,k,d]_q$ codes. Let us first give some definitions and a simple observation. We only consider \textit{non-degenerate} $[n,k]_q$ codes.  

\begin{defn}
Let $m_1,\ldots,m_k$ be elements in $\mathbb{F}_q$ not all equal to zero. The set $\mathcal{H}$ consisting of all vectors $X=(x_1,\ldots, x_k)$ such that $$m_1x_1+\cdots+ m_kx_k=c ~~\mbox{,\quad for } c\in \mathbb{F}_q\;,$$ 
is called a \textit{hyperplane}, which is a $(k-1)$-dimensional subspace of $(\mathbb{F}_q)^k$. 
\end{defn}
The number of the $1$-dimensional vector spaces in $(\mathbb{F}_{q})^k$ is equal to $\frac{q^k-1}{q-1}=q_k$ which coincides with the number of $(k-1)$-dimensional subspaces (hyperplanes) of $(\mathbb{F}_{q})^k$. Consequently, every hyperplane contains $\frac{q^{k-1}-1}{q-1}=q_{k-1}$ $k$-vectors over $\mathbb{F}_q$. Any pair of distinct hyperplanes in $(\mathbb{F}_{q})^k$ intersects in a $(k-2)$-dimensional subspace over $\mathbb{F}_q$. 
\begin{defn}\label{defn:char}
Let $C$ be an $[n,k]_q$ code with generator matrix $G$ where $M$ is the (multi)set of columns of $G$ and $A$ is a subspace in $(\mathbb{F}_q)^k$. Then $\mbox{Char}_G(A)$ is the number, including multiplicity, of $k$-vectors in the (multi)set $M\cap A $. We denote by $m(v)$ the multiplicity of the vector $v$ in $M$.
\end{defn}

\begin{rem}\label{remark-weght,char}
 Let $G$ be a generating matrix of an $[n,k]_q$ code $C$. For any non-zero vector $m=(m_1,\ldots,m_k)\in (\mathbb{F}_q)^k$, the hyperplane $m_1x_1+\cdots+m_kx_k=0$ contains $n-s$ columns (with multiplicity) of $G$ if and only if the codeword $mG$ has weight $s$. So we have a hyperplane with $\mbox{Char}_G(H)=n-s$ if and only if there is a codeword $c\in C$ with weight $s$. 
 
\end{rem}
\begin{theorem}
There exists an $[n,k,d]_q$ MWS code for each prime power $q$ and $k\geq 2$, where
$$n=2^{q_k}-1~~\mbox{ and }~~d=2^{q^{k-1}-1}-1.$$
\end{theorem}
\begin{proof}
The geometric construction given in \cite[Theorem 3.4]{alderson-neri2019maximum} leads to an $[n,k,d]_q$ code of length $n=2^{q_k}-1$. In the proof the characters of different hyperplanes are ranged from $2^{q_{k-1}}-1$ to $2^{q_k}-2^{q^{k-1}-1}$. So the minimum distance is $n-2^{q_k}-2^{q^{k-1}-1}=2^{q^{k-1}-1}-1$.
\end{proof}

\begin{lem}\label{known-lowbound}
If $C$ is an $[n,k,d]_q$ MWS code with $k\geq 2$, then 
$$n\geq \lceil\frac{q}{2}q_k\rceil~~,~~ d\geq \lceil \frac{q}{2}q_k-q_k+1\rceil~~\mbox{and}~~\Delta(C)\geq 0.$$
\end{lem}
\begin{proof}
The lower bound for $n$ was proven in \cite[Lemma 5.1]{alderson-neri2019maximum} and the bound for minimum distance is already given in Proposition \ref{distMWS-prop1}. Finally the bound for $\Delta(C)$ is a direct consequence of Theorem \ref{theorem: length spread}. 
\end{proof}
\begin{prop}
For $k=2$ the bounds in Lemma \ref{known-lowbound} are tight for all prime powers $q$. 
\end{prop}
\begin{proof}
The proof follows from Definition \ref{def: q 2 MWS}, Proposition \ref{prop-example SCMWS} and \cite[Proposition 5.4]{alderson-neri2019maximum}. 
\end{proof}
\begin{rem}
Let $\beta\in\mathbb{F}_q$ and let $c=(c_1,\ldots,c_n)$ be a vector in $(\mathbb{F}_q)^n$. The number of coordinates of $c$ equal to $\beta$ is denoted by $c[\beta]$, namely
$
c[\beta]=\left|
\{
i\in\{1,\ldots,n\}\mid c_i=\beta
\}
\right|\;$.
In \cite{alderson-neri2019maximum} the authors considered codes with the following property:
\begin{equation}\label{eq: property}
\mbox{There exists } \beta\in\mathbb{F}_q,\, \beta\neq0, 
\mbox{ such that, for } a,b\in C,\; a[\beta] = b[\beta]\mbox{ only if } a = b.
\tag{A}
\end{equation}
Due to \cite[Corollary 5.2]{alderson-neri2019maximum}, if an $[n,k,d]_q$ MWS code $C$ satisfy property \eqref{eq: property}, then $n\geq \lceil \frac{q\cdot q_{k+1}}{2} \rceil$. This follows by applying the bound in the Lemma \ref{known-lowbound} to the $[2n+1,k+1,d']_q$ MWS code $\bar{C}$ arisen from the construction given in \cite[Proposition 4.1]{alderson-neri2019maximum}. Using the same strategy we can get $$\Delta(c)\geq \dfrac{(q-2)(q^k-q)(q^{k+1}-1)}{4(q-2)^2\cdot q}.$$
in this setting the $q_{k}$ smallest elements in $S(\bar{C})$ are exactly the elements in $S(C)$. So the minimum distance of the new $[2n+1,k+1,d']_q$ MWS code $\bar{C}$ coincides with the minimum distance of $[n,k,d]_q$ code $C$ which means $d'=d$.
\end{rem}
\begin{prop}
There exist an $[7,3]_2$ strictly compact MWS code $C$, and there exists an $[32,3]_3$ MWS code $C'$ which is not strictly compact. 
\end{prop}
\begin{proof}
The generator matrix of an $[7,3]_2$ is a matrix $G\in \mathbb{F}_2^{3\times 7}$ where the (multi)set of columns $M$ can be generated by $3$ linearly independent vectors $v_1,v_2,v_3$ in $\mathbb{F}_2^3$ where $m(v_i)=2^i$. $\mathbb{F}_2^3$ contains  $q_3=7$ hyperplanes ($2$-dimensional subspaces) with characters $\{0,1,\ldots,6\}$. Due to Remark  \ref{remark-weght,char}, the set of nonzero weights is $S(C)=\{1,2,\ldots,7\}$. It is easy to verify that $d=1$, $\Delta(C)=0$, $|S(C)|=q_k$ and $n\in S(C)$. So $C$ is a strictly compact MWS code. The existence of $[7,3]_2$ MWS code is also shown in \cite{shi2019many,alderson-neri2019maximum}.


The second part was also proven in \cite{alderson-neri2019maximum}. Using the set of characters of hyperplanes, we can determine $S(C')=\{10,14,16,19,20,21,22,24,21,20,26,27,28,30,31\}$, $d=10$ and $\Delta(C')=50$. All the parameters satisfy the bounds given in the Lemma \ref{known-lowbound}. Moreover, we already proved that there is no strictly compact MWS code when $q.k$ is odd. 
\end{proof}



\begin{prop}\label{MWS bound K>=2}
For each $k\geq 2$ there exists an MWS code $C$ of length, minimum distance and spread  
\begin{equation*}
    n=q_{k-1}\binom{q_k}{2} ~~,~~d=q^{k-2}\left[\binom{q_k}{2}-q_k+1\right]~~\mbox{and }~~ \Delta(C)=\dfrac{q_k\cdot q_{k-1}(q_k\cdot q_{k-1}+q_{k-1}-q)}{2}.
\end{equation*}
\end{prop}
\begin{proof}
Let $\{H_0,\ldots, H_{q_k-1}\}$ be the set of hyperplanes in $(\mathbb{F}_{q})^k$. Define the generating matrix $G$ as follows. For each vector $v\in \mathbb{F}_q^k$, let $\mbox{Char}_G(v)=\sum_{v\in H_i}i$. For  $k=2$, a hyperplane is just a single vector of length $k$ over $\mathbb{F}_q$, so two hyperplanes might coincide or disjoint. If $k=3$, then two distinct hyperplanes  have intersection in a $1$-dimensional subspace. So for $k\geq 3$, a pair of distinct hyperplanes have intersection in a $k-2$-subspace. As a result, for $k\geq 2$ and $0\leq s\leq \frac{q^k-1}{q-1}-1=q_k-1$ we have 

\begin{equation*}
    \mbox{Char}_G(H_s)=q_{k-2}\binom{q_k}{2}+(q_{k-1}-q_{k-2})\cdot s,
\end{equation*}
which tells the minimum distance is $d=n-\mbox{Char}_G(H_{q_k-1})=q^{k-2}\left[\binom{q_k}{2}-q_k+1\right]$ and the number of columns of $G$ should be $n=q_{k-1}\binom{q_k}{2}$. The rest follows by applying  Theorem \ref{theorem: length spread}.
\end{proof}

The above length was given in \cite[Proposition 3.3]{alderson2019note} and an upper bound $n< q^{\frac{k^2+k-4}{2}}$ for the length of MWS codes of dimension $k\geq 3$ is given in \cite[Corollary 5.9]{alderson-neri2019maximum} which only gives a shorter length than the length in  Proposition \ref{MWS bound K>=2} where $k=3$.


\section{Conclusions}
In this work we introduce the notion of spread of a code, a tool to study the fundamental parameters of a code w.r.t the weight distribution of a target code. More precisely, the spread is a measure of how much the weight distribution of a a code $C$ is distant from the weight distribution of a target code. We focus here on MWS codes, a class of codes studied in the past few years by several authors, and we apply our methods to study the parameters of known examples of MWS codes.  As a result of our approach, we are able to completely characterise the length of MWS codes according to the their spread and to provide bounds on their minimum distances (Proposition \ref{distMWS-prop1}). Moreover, we specialise our results to two sub-families, namely to compact (Corollary \ref{cor: parameters compact}) and strictly-compact  (Theorem \ref{theorem: parameters}) MWS codes. We believe that the results obtained for MWS codes are a hint for the usefulness of analysing the parameters of families of codes according to their spread.















\subsection*{Acknowledgements}
The authors would like to thank Dr. Chunlei Li for his helpful advice and comments. 






\bibliographystyle{plain}

\bibliography{Refs}
\end{document}